\documentclass[12pt,twoside,leqno,openany]{amsart}
\usepackage{amssymb,amsbsy,amsmath,amsfonts,amscd,times}
\usepackage{graphics,color,xy,xypic,footmisc,fancyhdr,multicol}
\usepackage{fancybox,graphicx,mathrsfs}
\usepackage[T1]{fontenc}
\usepackage{breqn}
\newtheorem{theorem}{Theorem}
\newtheorem{lemma}{Lemma}

\newcommand{\C}{\mathbb{C}}

\newcommand{\R}{\mathbb{R}}

\newcommand{\ov}{\overline}
\newcommand{\LL}{\mathcal{L}}
\newcommand{\Lb}{\overline{\mathcal{L}}}
\newcommand{\T}{\mathcal{T}}
\newcommand{\s}{\mathcal{S}}

\newcommand{\Sb}{\ov{\mathcal{S}}}
\newcommand{\zb}{\ov{z}}
\let\mathcal\mathscr
\newcommand{\A}{{\sf a}}
\newcommand{\ab}{\ov{{\sf a}}}
\newcommand{\bb}{{\sf b}}
\newcommand{\bbb}{\ov{{\sf b}}}
\newcommand{\cc}{{\sf c}}

\newcommand{\dd}{{\sf d}}
\newcommand{\db}{\ov{{\sf d}}}
\newcommand{\ee}{{\sf e}}

\newcommand{\smallbullet}{{\scriptscriptstyle{\bullet}}}

\begin{document}

\setcounter{page}{80}

\title[]{
Canonical Cartan connection for $4$-dimensional CR-manifolds belonging to general class ${\sf II}$}

\author{Samuel Pocchiola}
\address{Samuel Pocchiola ---  D\'epartement de math\'ematiques, b\^atiment 425, Facult\'e des sciences d'Orsay,
Universit\'e Paris-Sud, F-91405 Orsay Cedex, france}
\email{samuel.pocchiola@math.u-psud.fr}

\maketitle

\bigskip
\section*{abstract}
We study the equivalence problem for $4$-dimensional
CR-manifolds of CR-dimension $1$ and codimension $2$ which have been referred to as belonging
to general class ${\sf II}$ in \cite{MPS}, and which are also known as Engel CR-manifolds.
We construct a canonical Cartan connection on such CR-manifolds through Cartan equivalence's method, thus providing an alternative approach
to the results contained 
in \cite{BES-2007}. In particular, we give the explicit expression of $4$ biholomorphic invariants, the annulation of which
is a necessary and sufficient condition for an Engel manifold to be locally biholomorphic to Beloshapka's cubic in $\C^3$.

\section{Introduction}
As highlighted by Henri Poincar\'e \cite{Poincare} in 1907, the (local) biholomorphic equivalence problem between two submanifolds
$M$ and $M^{\prime}$ of $\C^N$
is to determine whether or not there exists a (local) biholomorphism $\phi$ of $\C^N$ such that $\phi(M) = M^{\prime}$.
Elie Cartan \cite{Cartan-1932, Cartan-1933} solved this problem for 
hypersurfaces $M^3 \subset \C^2$ in 1932, as he constructed a ``hyperspherical connection'' on such hypersurfaces
by using the powerful technique which is now referred to as Cartan's equivalence method.

Given a manifold $M$ and some geometric data specified on $M$, which usually appears as a $G$-structure on $M$ (i.e. a reduction of the bundle of coframes of $M$),
Cartan's equivalence method seeks to provide a principal bundle $P$ on $M$ together with a coframe $\omega$ of $1$-forms on $P$ which
is adapted to the geometric structure of $M$ in the following sense:
an isomorphism between two such
geometric structures $M$ and $M^{\prime}$ lifts to a unique isomorphism between $P$ and $P^{\prime}$ which sends $\omega$ on $\omega^{\prime}$.
The equivalence problem between $M$ and $M^{\prime}$ 
is thus reduced to an equivalence problem between $\{e\}$-structures, which is well understood \cite{Olver-1995, Sternberg}.

We recall that a  CR-manifold $M$ is a real manifold endowed with a 
subbundle $L$ of $\C \otimes TM$ of even rank $2n$ such that
\begin{enumerate}
\item{$L \cap \ov{L}$ = \{0\}}
\item{$L$ is formally integrable, i.e. $\big[ L, \, L \big] \subset L$}.
\end{enumerate}
The integer $n$ is the CR-dimension of $M$ and $k = \dim M - 2n$ is the codimension of $M$. 
In a recent attempt \cite{MPS} to solve the equivalence problem for CR-manifolds up to dimension $5$, it has been shown  
that one can restrict the study to six different general classes of CR-manifolds of dimension $\leq 5$, 
which have been referred to as general classes 
${\sf I}$, ${\sf II}$, ${\sf III}_1$, ${\sf III}_2$, ${\sf IV}_1$ and ${\sf IV}_2$. 
The aim of this paper is to provide a solution to the equivalence problem for CR-manifolds which belong to general class ${\sf II}$,
that is the CR-manifolds of dimension $4$ and of CR-dimension $1$ whose CR-bundle $L$ satisfy the additional non-degeneracy condition:
\begin{equation*}
\C \otimes TM = L + \ov{L} + \big[ L, \, \ov{L} \big] + \big[ L, \big[ L, \, \ov{L} \big] \big],
\end{equation*}
meaning that $\C \otimes TM$ is spanned by $L$, $\ov{L}$ and their Lie brackets up to order $3$.

This problem has already been solved by Beloshapka, Ezhov and Schmalz in \cite{BES-2007}, 
where the CR-manifolds we study
are called Engel manifolds. The present paper provides thus an alternative solution to the results contained in \cite{BES-2007}.
The main result is the following:
\begin{theorem}\label{thm:intro}
Let $M$ be a CR-manifold belonging to general class ${\sf II}$.
There exists a 5-dimensional subbundle $P$ of the bundle of coframes $\C \otimes F(M)$ of $M$ and a coframe 
$\omega:=(\Lambda, \sigma, \rho, \zeta, \ov{\zeta})$ on $P$ such that any CR-diffeomorphism $h$ of $M$ lifts to a bundle isomorphism
$h^*$ of $P$ which satisfy $h^*(\omega) = \omega$. Moreover the structure equations of $\omega$ on $P$ are of the form:
\begin{equation*}
\begin{aligned}
d \sigma & = 3 \, \Lambda \wedge \sigma +
\rho \wedge \zeta
+
\rho \wedge \ov{\zeta}, \\
d \rho &=  2 \, \Lambda \wedge \rho + i \, \zeta \wedge \ov{\zeta} \\
d \zeta &= \Lambda \wedge \zeta + \mathfrak{I}_1 \, \sigma \wedge \rho + \mathfrak{I}_2 \, \sigma \wedge \zeta + 
\mathfrak{I}_3 \, \sigma 
\wedge \ov{\zeta} + \mathfrak{I}_4 \, \rho \wedge \zeta + \mathfrak{I}_5 \, \rho \wedge \ov{\zeta}, \\
d \ov{\zeta} &= \Lambda \wedge \ov{\zeta} + \ov{\mathfrak{I}_1} \, \sigma \wedge \rho + 
\ov{\mathfrak{I}_3} \, \sigma \wedge \zeta + 
\ov{\mathfrak{I}_2} \, \sigma 
\wedge \ov{\zeta} + \ov{\mathfrak{I}_5} \, \rho \wedge \zeta + \ov{\mathfrak{I}_4} \, \rho \wedge \ov{\zeta}, \\
d \Lambda & = \frac{i}{2} \mathfrak{ I}_1 \left.\sigma \wedge \ov{\zeta} \right.
- \frac{i}{2} \, \ov{\mathfrak{I}_1}  \left. \sigma \wedge \zeta \right.
- \frac{1}{3}  \left( \mathfrak{I}_2 + \ov{\mathfrak{I}_3 } \right)\left.\rho \wedge \zeta \right.
- \frac{1}{3}  \left( \ov{\mathfrak{I}_2} + \mathfrak{I}_3 \right) \left. \rho \wedge \zeta \right. \\
& + \mathfrak{I}_0 \left. \sigma \wedge \zeta \right., 
\end{aligned}
\end{equation*}
where $\mathfrak{I}_0$, $\mathfrak{I}_1$, $\mathfrak{I}_2$, $\mathfrak{I}_3$,  $\mathfrak{I}_4$, $\mathfrak{I}_5$,
are functions on $P$.
\end{theorem}

An example of CR-manifold belonging to general class ${\sf II}$ is provided by
Beloshapka's cubic ${\sf B} \subset \C^3$, which is defined by the equations:
\begin{equation*}
{\sf B}: \qquad \qquad
\begin{aligned}
w_1 & = \ov{w_1} + 2 \, i \, z \ov{z}, \\
w_2 & = \ov{w_2} + 2 \, i \, z \ov{z} \left( z + \ov{z} \right).
\end{aligned}
\end{equation*}
Cartan's equivalence method has been applied to Beloshapka's cubic in \cite{pocchiola2} where it has been shown that the coframe
$(\Lambda, \sigma, \rho, \zeta, \ov{\zeta})$ of theorem \ref{thm:intro} satisfy the simplified structure equations:
\begin{equation*}
\begin{aligned}
d \sigma &= 3 \left. \Lambda \wedge \sigma \right. + \left. \rho \wedge \zeta \right. 
+ \left.  \rho \wedge \ov{\zeta} \right., \\
d \rho &  =  2 \left. \Lambda \wedge \rho \right. + i \, \left. \zeta \wedge \ov{\zeta} \right., \\
d \zeta &= \left. \Lambda \wedge \zeta \right., \\
d \ov{\zeta} &= \left. \Lambda \wedge \ov{\zeta} \right., \\
d \Lambda & = 0,
\end{aligned}
\end{equation*}
corresponding to the case where the biholomorphic invariants $\mathfrak{I}_i$ vanish identically.
From this result together with theorem \ref{thm:intro}, we  deduce the existence of a Cartan connection on CR-manifolds belonging to general class {\sf II}
in section \ref{connection}.

We start in section \ref{G-structure} with the construction of a canonical $G$-structure $P^1$ on $M$, 
(e.g. a subbundle of the bundle
of coframes of $M$), 
which encodes the equivalence problem for $M$ under CR-automorphisms in the following sense:
a diffeomorphism 
\begin{equation*}
h: M \longrightarrow M
\end{equation*}
is a CR-automorphism of $M$ 
if and only if  
\begin{equation*}
h^* :P^1 \longrightarrow P^1
\end{equation*}
is a $G$-structure isomorphism of $P^1$.
We refer to \cite{MPS, Merker-2013-III, Merker-2013-IV} for details on the results summarized in this section
and to \cite{Sternberg} for an introduction to $G$-structures.
Section \ref{P1} is devoted to reduce successively $P^1$ to three subbundles:
\begin{equation*}
P^4 \subset P^3 \subset P^2 \subset P^1,
\end{equation*}
which are still adapted to the biholomorphic equivalence problem for $M$. We use Cartan equivalence method, for which we refer to \cite{Olver-1995}.
Eventually a Cartan connection is constructed on $P^4$ in section \ref{connection}.

\section{Initial G-structure} \label{G-structure}
Let $M$ be a $4$-dimensional CR-manifold belonging to general class ${\sf II}$ and $\LL$
be  a local generator of the CR-bundle $L$ of $M$. As $M$ belongs to general class ${\sf II}$,
the two vector fields $\T$, $\s$, defined by: 
\begin{equation*}
\begin{aligned}
\T & := i \, \big[ \LL, \Lb \big], \\
\s & := \big[ \LL, \T \big],\\
\end{aligned}
\end{equation*}
are such that:
\begin{equation*}
4 =  \text{rank}_{\C} \left(\LL, \Lb, \T , \s \right),
\end{equation*}
namely
\begin{equation*}
 \left(\LL, \Lb, \T , \s \right) \text{is a frame on $M$.}
\end{equation*}
As a result there exist two functions $A$ and $B$ such that:
\begin{equation*}
\Sb = A \cdot \T + B \cdot \s.
\end{equation*}
From the fact that $ \ov{\Sb} = \s,$ the functions $A$ and $B$ satisfy the relations:
\begin{equation} \label{eq:conjugate1}
\begin{aligned}
B \ov{B} & = 1, \\
\ov{A} +  \ov{ B} A &= 0.
\end{aligned}
\end{equation}
There also exist two functions $P$, $Q$
such that:
\begin{equation*}
\big[\LL,\s \big]
 = P \cdot \T + Q \cdot \s. 
\end{equation*}
The conjugate of $P$ and $Q$, $\ov{P}$ and $\ov{Q}$, are given by the relations:
\begin{equation} \label{eq:conjugate2}
\begin{aligned}
\ov{Q} & = \LL(B) + B \, Q + 2 A + \frac{\Lb(B)}{B}, \\
\ov{P} & = B \, \LL(A) - A \, \LL(B) - B \, A\, Q -A^2 - A \, \frac{\Lb(B)}{B} + \Lb(A) + B^2 \, P.
\end{aligned}
\end{equation}
The four functions $A$, $B$, $P$, $Q$ appear to be fundamental as all other Lie brackets between the 
vector fields $\LL$, $\Lb$, $\T$ and $\s$ are expressed in terms of these five functions and their
$\{ \LL, \Lb\}$-derivatives (\cite{Merker-2013-IV}).

In the case of an embedded CR-manifold $M \subset \C^3$, we can give an explicit formula for the fundamental vector field $\LL$, and hence
for the functions  $A$, $B$, $P$, $Q$, in terms of 
a graphing function of $M$. We refer to \cite{Merker-2013-V} for details on this question.
Let us just mention that the submanifold $M \subset \C^3$ is represented in local coordinates:
\begin{equation*}
(z,w_1,w_2) := (x + i \, y, \, u_1 + i\, v_1, \,u_2 + i\, v_2)
\end{equation*}
as a graph:
\begin{equation*}
\begin{aligned}
v_1 &= \phi_1(x, y, u_1, u_2)\\
v_2 &= \phi_2(x, y, u_1, u_2).
\end{aligned}
\end{equation*}
There exists then a unique local generator $\LL$ of $T^{1,0}M$ of the form:
\begin{equation*}
\LL= \frac{\partial}{\partial z} 
+ A^1 \, \frac{\partial}{\partial u_1} 
+  A^2 \, \frac{\partial}{\partial u_2} 
\end{equation*}
having conjugate:
\begin{equation*}
\Lb= \frac{\partial}{\partial \zb} 
+ \ov{A^1} \, \frac{\partial}{\partial u_1} 
+  \ov{A^2} \, \frac{\partial}{\partial u_2} 
\end{equation*}
which is a generator of $T^{0,1}M$, where
the functions $A^1$ and $A^2$ are given by the determinants:
\begin{equation*}
A^1:= \frac{
\begin{vmatrix}
- \phi_{1,z} & \phi_{1, u_2} \\
- \phi_{2,z} & i + \phi_{2, u_2}
\end{vmatrix}
}
{\begin{vmatrix}
i + \phi_{1,u_1} & \phi_{1, u_2} \\
\phi_{2, u_1} & i + \phi_{2, u_2}
\end{vmatrix}
},
\qquad \qquad 
A^2:= \frac{
\begin{vmatrix}
 i + \phi_{1, u_1} & - \phi_{1,z}  \\
 \phi_{2, u_1} & - \phi_{2,z} 
\end{vmatrix}
}
{\begin{vmatrix}
i + \phi_{1,u_1} & \phi_{1, u_2} \\
\phi_{2, u_1} & i + \phi_{2, u_2}
\end{vmatrix}
}
.\end{equation*}

Returning to the general case of abstract CR-manifolds, let us introduce the coframe
\begin{equation*}
\omega_0 :=\left(\sigma_0, \rho_0, \zeta_0, \ov{\zeta}_0 \right),
\end{equation*}
as the dual coframe of 
$\left(\s, \T, \LL, \Lb \right)$.  We have {\cite{Merker-2013-IV}:

\begin{lemma}
The structure equations enjoyed by $\omega_0$ are of the form: 
\begin{equation*}
\begin{aligned}
d \sigma_0 & =H \, 
\sigma_0\wedge \rho_0 + F \, \sigma_0\wedge \ov{\zeta}_0 + 
Q \, \sigma_0\wedge\zeta_0 + B \,  \rho_0\wedge \ov{\zeta}_0 + \rho_0\wedge\zeta_0, \\
d \rho_0 & = G \, \sigma_0\wedge \rho_0 + E \, \sigma_0\wedge \ov{\zeta}_0 
+ P \, \sigma_0\wedge\zeta_0 + A \, \rho_0\wedge \ov{\zeta}_0 + i \, \zeta_0\wedge \ov{\zeta}_0, \\
d \zeta_0 &= 0, \\
d \ov{\zeta}_0 &= 0,
\end{aligned}
\end{equation*}
where the four functions:
\begin{equation*}
E, \, F, \, G, \, H,
\end{equation*}
can be expressed 
in terms of the four fundamental functions:
\begin{equation*}
A,\, B,\, P,\, Q,
\end{equation*}
and their $\{ \LL, \Lb\}$-derivatives as:
\begin{equation*}
\begin{aligned}
E &:= \LL(A) + B \,  P,\\
F &:= \LL(B) + B \, Q + A,\\
G &:=i \,\LL(\LL(A)) +i \, P \, \LL(B) - i \, \LL(P) - i \, Q \, \LL(A) + i \, P \, \LL(B) + i \,B \LL(P),\\
H &:=i \,\LL(\LL(B)) +i \,Q \, \LL(B) +i \, B \, \LL(Q) + 2i \, \LL(A) - i \, \LL(Q).
\end{aligned}
\end{equation*}

\end{lemma}

Let $h: \, M \longrightarrow M$ be a CR-automorphism of $M$. As we have 
\begin{equation*}
h_* \left( L \right) = L,  
\end{equation*}
there exists a non-vanishing complex-valued function $\A$ on $M$ such that:
\begin{equation*}
h_* \left( \LL \right) = \A \, \LL.
\end{equation*}
From the definition of $\T$, $\s$,  and the invariance 
\begin{equation*}
h_* \left( \big[ X, Y \big] \right) = \big[ h_*(X), h_*(Y) \big]
\end{equation*}
for any vector fields $X$, $Y$ on $M$, we easily get the existence of four functions
\begin{equation*}
\bb, \cc , \dd , \ee : M \longrightarrow \C,
\end{equation*}
such that:
\begin{equation*}
h_*
\begin{pmatrix}
\LL \\
\Lb \\
\T \\
\s 
\end{pmatrix}
=
\begin{pmatrix}
\A & 0 & 0 & 0  \\
 0 & \ab & 0 & 0  \\
\bb & \bbb & \A \ab & 0  \\
\ee & \dd & \cc & \A^2 \ab  \\
\end{pmatrix}
\cdot
\begin{pmatrix}
\LL \\
\Lb \\
\T \\
 \s \\
\end{pmatrix}.
\end{equation*}

This is summarized in the following lemma {\cite{Merker-2013-III}}:
\begin{lemma} \label{lemma:matrix}
Let $h: \, M \longrightarrow M$ a CR-automorphism of $M$ and let $G_1$ 
be the subgroup of ${\sf GL}_4(\C)$
\begin{equation*}
G_1:= \left\{
\begin{pmatrix}
{\A^2} \ab & 0 & 0 & 0  \\
\cc & \A \ab & 0 & 0 \\
\dd & \bb & \A & 0 \\
\ee & \bbb & 0 & \ab
\end{pmatrix}
, \, 
\A \in \C \setminus{\{0\}}, \,
\bb, \cc, \dd, \ee \in \C \right \}.
\end{equation*}
Then the pullback $\omega$ of $\omega_0$ by $h$,
$\omega:= h^* \omega_0,$
satisfies:
\begin{equation*}
\omega = g \cdot \omega_0,
\end{equation*}
where $g$ is smooth (locally defined) function 
$M \stackrel{g}{\longrightarrow} G_1$.
\end{lemma}

This motivates the introduction of the subbundle $P^1$ of the bundle of coframes on $M$
constituted by the coframes $\omega$ of the form
\begin{equation*}
\omega:= g \cdot \omega_0, \qquad g \in G_1.
\end{equation*}
The next section is devoted to reduce successively $P^1$ to three subbundles:
\begin{equation*}
P^4 \subset P^3 \subset P^2 \subset P^1,
\end{equation*}
which are adapted to the biholomorphic equivalence problem for $M$.

\section{Reductions of $P^1$}\label{P1}
The coframe $\omega_0$ gives a natural (local) trivialisation $P^1 \stackrel{tr} \longrightarrow M \times G_{1}$ from which
we may consider any differential form on $M$  
(resp. $G_1$) as a differential form on $P^1$ through the pullback by the first
(resp. the second) component of $tr$.
With this identification, 
the structure equations of $P^1$ are naturally obtained by the formula:
\begin{equation} 
\label{eq:lifted}
d \omega =  dg \cdot g^{-1} \wedge \omega + g \cdot d \omega_{0}.
\end{equation}
The term $ g \cdot d \omega_{0}$  contains the so-called torsion coefficients of $P^1$.
A $1$-form $\widetilde{\alpha}$ on $P^1$ is called a modified Maurer-Cartan form if its restriction to any fiber of $P^1$ 
is a Maurer-Cartan form of $G_1$, or equivalently, if it is of the form:
\begin{equation*}
\widetilde{\alpha} := \alpha - x_{\sigma} \, \sigma  -x_{\rho} \, \rho 
- x_{\zeta} \, \zeta  -
x_{\overline{\zeta}} \, \overline{\zeta},
\end{equation*}
where $x_{\sigma}$, $x_{\rho}$, $x_{\zeta}$, $x_{\ov{\zeta}}$, are arbitrary complex-valued functions on $M$ and
where $\alpha$
is a Maurer-Cartan form of $G_1$.

A basis for the Maurer-Cartan forms of $G_1$ is given by the following $1$-forms:
\begin{equation*}
\begin{aligned}
\alpha^1  &: = {\frac {{ d\A}}{\A}}, \\
\alpha^2 &:= -{\frac {\bb{ d\A}}{{\A}^{2}{ \ab}}}+{\frac {{ d\bb}}{\A{ \ab}}}, \\
\alpha^3  &:= -{\frac {\cc{ d\A}}{{ \ab}\,{\A}^{3}}}-{\frac {\cc{ d\ab
}}{{{ \ab}}^{2}{\A}^{2}}} + { \frac {{ d\cc}}{{\A}^{2}{ \ab}}}, \\
\alpha^4 & =-{\frac { \left( \dd\A{ \ab}-\bb\cc \right) { d\A}}{{\A}^{4}{
{ \ab}}^{2}}}-{\frac {\cc{ d\bb}}{{\A}^{3}{{ \ab}}^{2}}}+{\frac {{
 d\dd}}{{\A}^{2}{ \ab}}}, \\
\alpha^5 & =-{\frac { \left( \ee\A{ \ab}-{ \bbb}\,\cc \right) { d\ab}
}{{\A}^{3}{{ \ab}}^{3}}}-{\frac {\cc{ d\bbb}}{{\A}^{3}{{ \ab}}^{2}}}+
{\frac {{ d\ee}}{{\A}^{2}{ \ab}}}, 
\end{aligned}
\end{equation*}
together with their conjugate.

We derive the structure equations of $P^1$ from the relations (\ref{eq:lifted}), from which we extract the expression 
of $d \sigma$: 
\begin{multline*}
d \sigma = 2 \left.\alpha^1 \wedge \sigma \right. + \left.\ov{\alpha^1} \wedge \sigma \right. \\
+ T^{\sigma}_{\sigma \rho} \left.\sigma \wedge \rho \right.
- T^{\sigma}_{\sigma \zeta} \left.\sigma \wedge \zeta \right. 
- T^{\sigma}_{\sigma \ov{\zeta}} \left. \sigma \wedge \ov{\zeta} \right. + \left.\rho \wedge \zeta \right. + 
\frac{\A }{\ab} \, B \, \left. \rho \wedge \ov{\zeta} \right.,
\end{multline*}
or equivalently:
\begin{equation*}
d \sigma = 2 \left. \widetilde{\alpha}^1 \wedge \sigma \right. + \left.\ov{\widetilde{\alpha}^1} \wedge \sigma \right. 
 + \left.\rho \wedge \zeta \right. + 
\frac{\A }{\ab} \, B \, \left. \rho \wedge \ov{\zeta} \right.,
\end{equation*}
for a modified Maurer-Cartan form $\widetilde{\alpha}^1$.
The coefficient 
\begin{equation*}
\frac{\A }{\ab} \, B, 
\end{equation*}
which can not be absorbed for any choice of the modified Maurer-Cartan form $\widetilde{\alpha}^1$,
is referred to as an essential torsion coefficient.
From standard results on Cartan theory (see \cite{Olver-1995, Sternberg}), a diffeomorphism of $M$
is an isomorphism of the $G_1$-structure $P^1$ if and only if it is an isomorphism of the reduced 
bundle $P^2 \subset P^1$ consisting
of those coframes $\omega$ on $M$ such that  
\begin{equation*}
\frac{\A }{\ab} \, B  = 1. 
\end{equation*}
This is equivalent to the normalization: 
\begin{equation*}
 \ab = \A B.
\end{equation*}

A coframe $\omega \in P^2$ is related to the coframe $\omega_0$ by the relations:
\begin{equation*}
\left\{ \begin{aligned}
\sigma &= \A^3 \, B \, \sigma_0 \\
\rho &= \cc \, \sigma_0 + \A^2 \, B \, \rho_0 \\
\zeta & = \dd \,  \sigma_0 + \bb \, \rho_0 + \A \, \zeta_0\\
\ov{\zeta} & = \ee \, \sigma_0 + \bbb \, \rho_0 + \A \, B \, \ov{\zeta}_0,
\end{aligned} \right.
\end{equation*}
which are equivalent to:
\begin{equation*}
\left\{ \begin{aligned}
\sigma &=\left. \A^{\prime} \right.^{3}  \sigma_1 \\
\rho &= \cc^{\prime} \, \sigma_1 + \left.\A^{\prime}\right.^2  \rho_1 \\
\zeta & = \dd^{\prime} \,  \sigma_1 + \bb \, \rho_1 + \A^{\prime} \, \zeta_1\\
\ov{\zeta} & = \ee^{\prime} \, \sigma_1 + \bbb \, \rho_1 + \A^{\prime} \,\ov{\zeta}_1,
\end{aligned} \right.
\end{equation*}
where:
\begin{equation*}
\sigma_1 : = \frac{\sigma_0}{B^{\frac{1}{2}}}, \qquad \qquad \rho_1 : = \rho_0, \qquad \qquad 
\zeta_1 : = \frac{\zeta_0}{B^{\frac{1}{2}}},
\end{equation*}
and 
\begin{equation*}
x^{\prime} := x \cdot B^{\frac{1}{2}}, \qquad \qquad \qquad \text{for}  \quad x = \A, \, \cc, \,  \dd, \, \ee.
\end{equation*}
We notice that $\A^{\prime}$ is a real parameter, and that $\sigma_{1}$ is a real $1$-form.
Let $\omega_1$ be the coframe $\omega_1:=\left(\sigma_1, \rho_1, \zeta_1, \ov{\zeta}_1 \right)$, 
and $G_2$ be the subgroup of $G_1$:
\begin{equation*}
G_2:= \left\{
\begin{pmatrix}
{\A^3} & 0 & 0 & 0  \\
\cc & \A^2 & 0 & 0 \\
\dd & \bb & \A & 0 \\
\ee & \bbb & 0 & \A
\end{pmatrix}
, \, 
\A \in \R \setminus \{0\}, \,
\bb, \cc, \dd, \ee \in \C \right \}.
\end{equation*}  
A coframe $\omega$ on $M$ belongs to $P^2$ if and only if there is a local function
$g: M \stackrel{g}{\longrightarrow}  G_2$ such that $\omega = g \cdot \omega_1$, namely $P^2$ is a $G_2$ structure on $M$.

The Maurer-Cartan forms of $G_2$ are given by:
\begin{equation*}
\begin{aligned}
\beta^1  &: = {\frac {{ d\A}}{\A}}, \\
\beta^2 &:= -{\frac {\bb d\A }{{\A}^{3}}}+{\frac {{ d\bb}}{\A^2}}, \\
\beta^3  &:=- 2 \, {\frac {\cc  d\A }{{\A}^{4}}} + { \frac {{ d\cc}}{{\A}^{3}}}, \\
\beta^4 & =-{\frac { \left( \dd\A^2 -\bb \cc \right) { d\A}}{{\A}^{6}}}-{\frac {\cc{ d\bb}}{{\A}^{5}}}+{\frac {{
d\dd}}{{\A}^{3}}}, \\
\beta^5 & =-\frac { \left( \ee\A^2-{ \bbb}\,\cc \right)  d\A
}{\A^6}-\frac {\cc{ d\bbb}}{\A^{5}}+
\frac { d\ee}{\A^{3}}, \\
\end{aligned}
\end{equation*}
together with $\ov{\beta^2}$, $\ov{\beta^3}$, $\ov{\beta^4}$, $\ov{\beta^5}$.
Using formula (\ref{eq:lifted}), we get the structure equations of $P^2$:

\begin{multline*}
d \sigma = 3 \, \beta^1 \wedge \sigma  \\
+
U^{\sigma}_{\sigma \rho} \left. \sigma \wedge \rho \right.
+
U^{\sigma}_{\sigma \zeta} \left. \sigma \wedge \zeta \right.
+
U^{\sigma}_{\sigma \ov{\zeta}} \left. \sigma \wedge \ov{\zeta} \right.
+
\rho \wedge \zeta
+
\rho \wedge \ov{\zeta}
\end{multline*}

\begin{multline*}
d \rho
=
2 \beta^{1} \wedge \rho + \beta^3 \wedge \sigma \\
+
U^{\rho}_{\sigma \rho} \, \sigma \wedge \rho
+
U^{\rho}_{\sigma \zeta} \, \sigma \wedge \zeta
+
U^{\rho}_{\sigma \ov{\zeta}} \, \sigma \wedge \ov{\zeta}  \\
+
U^{\rho}_{\rho \zeta} \, \rho \wedge \zeta  
+
U^{\rho}_{\rho \ov{\zeta}} \, \rho \wedge \ov{\zeta} 
+
i \, \zeta \wedge \ov{\zeta}
,\end{multline*}

\begin{multline*}
d \zeta
=
{\beta}^{1} \wedge \zeta + {\beta}^{2} \wedge \rho + {\beta}^{4} \wedge \sigma \\
+
U^{\zeta}_{\sigma \rho} \, \sigma \wedge \rho
+
U^{\zeta}_{\sigma \zeta} \, \sigma \wedge \zeta 
+
U^{\zeta}_{\sigma \ov{\zeta}} \, \sigma \wedge \ov{\zeta} 
+
U^{\zeta}_{\rho \zeta} \, \rho \wedge \zeta \\
+
U^{\zeta}_{\rho \overline{\zeta}} \, \rho \wedge \overline{\zeta}
+
U^{\zeta}_{\zeta \overline{\zeta}} \, \zeta \wedge \overline{\zeta}.
\end{multline*}

Introducing the modified Maurer-Cartan forms:
\begin{equation*}
\widetilde{\beta}^i= \beta^i  - y_{\sigma} \, \sigma - y_{\rho}^i \, \rho  - y_{\zeta}^i \, \zeta \, 
- y_{\overline{\zeta}}^i \, \overline{\zeta},
\end{equation*}
the structure equations rewrite:
\begin{multline*}
d \sigma = 3 \left. \widetilde{\beta}^1 \wedge \sigma \right. \\
+
\left( U^{\sigma}_{\sigma \rho} - 3 \, y^1_{\rho} \right) \, \left. \sigma \wedge \rho \right. 
+
\left( U^{\sigma}_{\sigma \zeta} -3 \,y^1_{\zeta} \right) \, \left. \sigma \wedge \zeta \right. \\
+
\left( U^{\sigma}_{\sigma \ov{\zeta}}  -3 \,y^1_{\ov{\zeta}} \right) \, \left. \sigma \wedge \ov{\zeta} \right. 
+ \rho \wedge \zeta
+
\rho \wedge \ov{\zeta}
\end{multline*}
\begin{multline*}
d \rho
=
2 \widetilde{\beta}^{1} \wedge \rho + \widetilde{\beta}^3 \wedge \sigma \\
+
\left( U^{\rho}_{\sigma \rho} + 2 \, y^1_{\sigma} - y^3_{\rho} \right) \, \left. \sigma \wedge \rho \right.
+
\left( U^{\rho}_{\sigma \zeta} - y^3_{\zeta} \right) \, \left. \sigma \wedge \zeta \right. \\
+
\left( U^{\rho}_{\sigma \ov{\zeta}} - y^3_{\ov{\zeta}} \right) \, \left. \sigma \wedge \ov{\zeta} \right.
+
\left( U^{\rho}_{\rho \zeta} - 2 \, y^1_{\zeta} \right) \, \left. \rho \wedge \zeta \right.  \\ 
+
\left( U^{\rho}_{\rho \ov{\zeta}} - 2 \, y^1_{\ov{\zeta}} \right) \, \left. \rho \wedge \ov{\zeta} \right.
+
i \, \left. \zeta \wedge \ov{\zeta} \right.
,\end{multline*}
\begin{multline*}
d \zeta
=
\widetilde{\beta^{1}} \wedge \zeta + \widetilde{\beta^{2}} \wedge \rho + \widetilde{\beta^{4}} \wedge \sigma \\
+
\left( U^{\zeta}_{\sigma \rho} + y^2_{\sigma}- y^4_{\rho} \right) \, \left. \sigma \wedge \rho \right.
+
\left( U^{\zeta}_{\sigma \zeta} + y^1_{\sigma} - y^4_{\zeta} \right) \, \left. \sigma \wedge \zeta \right. \\ 
+
\left( U^{\zeta}_{\sigma \ov{\zeta}} - y^4_{\ov{\zeta}} \right) \, \left. \sigma \wedge \ov{\zeta} \right. 
+
\left( U^{\zeta}_{\rho \zeta} + y^1_{\rho} - y^2_{\zeta} \right) \, \left. \rho \wedge \zeta \right. \\
+
\left( U^{\zeta}_{\rho \overline{\zeta}} - y^2_{\ov{\zeta}} \right) \, \left. \rho \wedge \overline{\zeta} \right.
+
\left( U^{\zeta}_{\zeta \overline{\zeta}} - y^1_{\ov{\zeta}} \right) \, \left. \zeta \wedge \overline{\zeta} \right.
,\end{multline*}
which leads to the following absorbtion equations:
\begin{alignat*}{3}
3 \, y^1_{\rho} & =  U^{\sigma}_{\sigma \rho}, & \qquad \qquad  
3 \,y^1_{\zeta} & =  U^{\sigma}_{\sigma \zeta}, & \qquad \qquad
3 \,y^1_{\ov{\zeta}} & =  U^{\sigma}_{\sigma \ov{\zeta}}, \\
-2 \, y^1_{\sigma} + y^3_{\rho}  &= U^{\rho}_{\sigma \rho}, &\qquad \qquad 
y^3_{\zeta}  &=  U^{\rho}_{\sigma \zeta}  , &\qquad \qquad   
 y^3_{\ov{\zeta}}   &= U^{\rho}_{\sigma \ov{\zeta}}, \\ 
2 \, y^1_{\zeta} &= U^{\rho}_{\rho \zeta}, & \qquad \qquad
2 \, y^1_{\ov{\zeta}}   &=   U^{\rho}_{\rho \ov{\zeta}}, & \qquad \qquad
-y^2_{\sigma} + y^4_{\rho}  & = U^{\zeta}_{\sigma \rho}, \\
-y^1_{\sigma} + y^4_{\zeta} & = U^{\zeta}_{\sigma \zeta}, & \qquad \qquad
y^4_{\ov{\zeta}} & = U^{\zeta}_{\sigma \ov{\zeta}}, & \qquad \qquad
-y^1_{\rho} + y^2_{\zeta} & = U^{\zeta}_{\rho \zeta}, \\
y^2_{\ov{\zeta}}  & = U^{\zeta}_{\rho \overline{\zeta}}, & \qquad \qquad
 y^1_{\ov{\zeta}} & =  U^{\zeta}_{\zeta \overline{\zeta}}.
\end{alignat*} 
Eliminating $y^1_{\ov{\zeta}}$ among these equations leads to:
\begin{equation*}
U^{\zeta}_{\zeta \ov{\zeta}} = \frac{1}{2} \, U^{\rho}_{\rho \ov{\zeta}} = \frac{1}{3} \, U^{\sigma}_{\sigma \ov{\zeta}} 
,\end{equation*}
from which we deduce the following normalizations:
\begin{equation*}
\cc = \A^2 \, {\bf C}_0,
\end{equation*}
and 
\begin{equation*}
\bb = \A \, {\bf B}_0,
\end{equation*} 
where:
\begin{equation*}
{\bf C}_0 := \left( \frac{1}{2} \, \frac{\LL(B)}{B^{\frac{1}{2}}} + \frac{1}{2} \,  Q B^{\frac{1}{2}} \right),
\end{equation*}
and
\begin{equation*}
{\bf B}_0 := \left( \frac{i}{3} \, \frac{\Lb(B)}{B^{\frac{3}{2}}} 
-  \frac{i}{3} \, \frac{A}{B^{\frac{1}{2}}} - \frac{i}{6} \, B^{\frac{1}{2}} Q - 
\frac{i}{6} \, \frac{\LL(B)}{B^{\frac{1}{2}}} \right)
.\end{equation*} 
We introduce the coframe $\omega_2:=\left(\sigma_2, \rho_2, \zeta_2, \ov{\zeta}_2 \right)$ on $M$, defined by:
\begin{equation*}
\left \{
\begin{aligned}
\sigma_2 & := \sigma_1, \\
\rho_2 &:= \rho_1 + {\bf C}_0 \, \sigma_1, \\
\zeta_2 &:= \zeta_1 + {\bf B}_0 \, \rho_1, 
\end{aligned} \right.
\end{equation*}
and the $3$-dimensional subgroup $G_3 \subset G_2$: 
\begin{equation*}
G_3:= \left\{
\begin{pmatrix}
{\A^3} & 0 & 0 & 0  \\
0 & \A^2 & 0 & 0 \\
\dd & 0 & \A & 0 \\
\db & 0  & 0 & \A
\end{pmatrix}
, \, 
\A \in \R \setminus \{0\}, \, \dd \in \C \right \}.
\end{equation*}
The normalizations:
\begin{equation*}
\bb:= \A \, {\bf B}_0, \qquad \qquad \cc:= \A^2 \, {\bf C}_0,
\end{equation*}
amount to
consider the subbundle $P^3 \subset P^2$ consisting of those coframes $\omega$ of the form
\begin{equation*}
\omega := g \cdot \omega_2, \qquad \text{where $g$ is a function} \,\,\,  
g: M \stackrel{g}{\longrightarrow}  G_3. 
\end{equation*}

A basis of the Maurer Cartan forms of $G_3$ is given by: 
\begin{equation*}
\gamma^1 : = {\frac {{ d\A}}{\A}}, \qquad \qquad
\gamma^2: =-{\frac {\dd { d\A}}{{\A}^{4}}} + {\frac {{
 d\dd}}{{\A}^{3}}}, \qquad \qquad \ov{\gamma}_2
.\end{equation*}

The structure equations of $P^3$ are:
\begin{multline*}
d \sigma = 3 \, \gamma^1 \wedge \sigma  \\
+
V^{\sigma}_{\sigma \rho} \left. \sigma \wedge \rho \right.
+
V^{\sigma}_{\sigma \zeta} \left. \sigma \wedge \zeta \right.
+
V^{\sigma}_{\sigma \ov{\zeta}} \left. \sigma \wedge \ov{\zeta} \right.
+
\rho \wedge \zeta
+
\rho \wedge \ov{\zeta}
,\end{multline*}
\begin{multline*}
d \rho
=
2 \gamma^{1} \wedge \rho 
+
V^{\rho}_{\sigma \rho} \, \sigma \wedge \rho
+
V^{\rho}_{\sigma \zeta} \, \sigma \wedge \zeta
+
V^{\rho}_{\sigma \ov{\zeta}} \, \sigma \wedge \ov{\zeta}  \\
+
V^{\rho}_{\rho \zeta} \, \rho \wedge \zeta  
+
V^{\rho}_{\rho \ov{\zeta}} \, \rho \wedge \ov{\zeta} 
+
i \, \zeta \wedge \ov{\zeta}
,\end{multline*}
\begin{multline*}
d \zeta
=
{\gamma}^{1} \wedge \zeta + {\gamma}^{2} \wedge \sigma \\
+
V^{\zeta}_{\sigma \rho} \, \sigma \wedge \rho
+
V^{\zeta}_{\sigma \zeta} \, \sigma \wedge \zeta 
+
V^{\zeta}_{\sigma \ov{\zeta}} \, \sigma \wedge \ov{\zeta} 
+
V^{\zeta}_{\rho \zeta} \, \rho \wedge \zeta \\
+
V^{\zeta}_{\rho \overline{\zeta}} \, \rho \wedge \overline{\zeta}
+
V^{\zeta}_{\zeta \overline{\zeta}} \, \zeta \wedge \overline{\zeta}
.\end{multline*}
\begin{multline*}
d \ov{\zeta}
=
{\gamma}^{1} \wedge \zeta + {\gamma}^{3} \wedge \sigma \\
+
V^{\ov{\zeta}}_{\sigma \rho} \, \sigma \wedge \rho
+
V^{\ov{\zeta}}_{\sigma \zeta} \, \sigma \wedge \zeta 
+
V^{\ov{\zeta}}_{\sigma \ov{\zeta}} \, \sigma \wedge \ov{\zeta} 
+
V^{\ov{\zeta}}_{\rho \zeta} \, \rho \wedge \zeta \\
+
V^{\ov{\zeta}}_{\rho \overline{\zeta}} \, \rho \wedge \overline{\zeta}
+
V^{\ov{\zeta}}_{\zeta \overline{\zeta}} \, \zeta \wedge \overline{\zeta}
.\end{multline*}

It is straightforward to 
notice that $V^{\rho}_{\sigma \zeta}$ and $V^{\rho}_{\sigma \ov{\zeta}}$ are two essential torsion coefficients. 
The first one leads to the normalization:
\begin{equation*}
\db= \A \, \ov{{\bf D}}_0,
\end{equation*}
with
\begin{equation*}
\ov{{\bf D}}_0 := \frac{i}{2} \, 
\frac{\LL(B)^2}{B} + \frac{i}{3} \, Q \LL(B)  - \frac{i}{2} \, \LL\left( \LL(B) \right) - \frac{i}{2} \, B \LL(Q) 
+ \frac{i}{2} \, A \frac{\LL(B)}{B} + \frac{i}{6} \, A Q + i B P,
\end{equation*}
while the second essential torsion coefficient gives the normalization:
\begin{equation*}
\dd = \A \, {\bf D}_0,
\end{equation*}
with:
\begin{multline*}
{\bf D}_0: = - \frac{2i}{3} \, \LL(B)\, Q - \frac{i}{6} \, \frac{\LL(B)\,  A}{B} - 
\frac{i}{6} \, A \, Q + \frac{i}{6} \, \frac{\Lb(B)\,Q}{B} - \frac{i}{3} \, \frac{\LL(B)^2}{B}
- \frac{i}{3} \, B \, Q^2 \\
- i \, \LL(A) - \frac{i}{3} \, \frac{\Lb(B) \, \LL(B)}{B^2} 
+ \frac{i}{2} \, \frac{\Lb\left(\LL(B)\right)}{B} + \frac{i}{2} \, \Lb(Q) - i\, B\, P. 
\end{multline*}
The coherency of the above formulae can be checked using the relations (\ref{eq:conjugate1}) and (\ref{eq:conjugate2}).

Let $G_4$ be the $1$-dimensional Lie subgroup of $G_3$ whose elements $g$ are of the form:
\begin{equation*}
g := \begin{pmatrix}
{\A}^3 & 0 & 0 & 0  \\
0  & \A^2 & 0 & 0 \\
0 & 0 & \A & 0 \\
0 & 0 & 0 & \A
\end{pmatrix}, \qquad \A \in \R \setminus \{ 0 \},
\end{equation*}
and let $\omega_3:=\left( \sigma_3, \rho_3, \zeta_3, \ov{\zeta}_3\right)$ 
be the coframe defined on $M$ by:
\begin{equation*}
\sigma_3 := \sigma_2, \qquad \qquad
\rho_3 : = \rho_2,  \qquad   \qquad
\zeta_3 : = \zeta_2 + {\bf D}_0 \, \sigma_2.
\end{equation*}
The normalization of $\dd$ is equivalent to the reduction
of $P^3$ to a subbundle $P^4$ consisting of those coframes $\omega$ on $M$ such that:
\begin{equation*}
\omega := g \cdot \omega_3, \qquad \text{where $g$ is a function} \,\,\,  
g: M \stackrel{g}{\longrightarrow}  G_4. 
\end{equation*}

The Maurer-Cartan forms of $G_4$ are spanned by: 
\begin{equation*}
\alpha:= \frac{d \A}{\A}.
\end{equation*}

Proceeding as in the previous steps, we compute the structure equations of $P^4$:
\begin{multline*}
d \sigma = 3 \, \frac{d \A}{\A} \wedge \sigma  \\
+
W^{\sigma}_{\sigma \rho} \left. \sigma \wedge \rho \right.
+
W^{\sigma}_{\sigma \zeta} \left. \sigma \wedge \zeta \right.
+
W^{\sigma}_{\sigma \ov{\zeta}} \left. \sigma \wedge \ov{\zeta} \right.
+
\rho \wedge \zeta
+
\rho \wedge \ov{\zeta}
,\end{multline*}
\begin{multline*}
d \rho
=
2 \, \frac{d \A}{\A} \wedge \rho 
+
W^{\rho}_{\sigma \rho} \, \sigma \wedge \rho
+
W^{\rho}_{\rho \zeta} \, \rho \wedge \zeta  
+
W^{\rho}_{\rho \ov{\zeta}} \, \rho \wedge \ov{\zeta} 
+
i \, \zeta \wedge \ov{\zeta}
,\end{multline*}
\begin{multline*}
d \zeta
=
\frac{d \A}{\A} \wedge \zeta \\
+
W^{\zeta}_{\sigma \rho} \, \sigma \wedge \rho
+
W^{\zeta}_{\sigma \zeta} \, \sigma \wedge \zeta 
+
W^{\zeta}_{\sigma \ov{\zeta}} \, \sigma \wedge \ov{\zeta} 
+
W^{\zeta}_{\rho \zeta} \, \rho \wedge \zeta \\
+
W^{\zeta}_{\rho \overline{\zeta}} \, \rho \wedge \overline{\zeta}
+
W^{\zeta}_{\zeta \overline{\zeta}} \, \zeta \wedge \overline{\zeta}
,\end{multline*}
\begin{multline*}
d \ov{\zeta}
=
\frac{d \A}{\A}  \wedge \ov{\zeta} \\
+
W^{\ov{\zeta}}_{\sigma \rho} \, \sigma \wedge \rho
+
W^{\ov{\zeta}}_{\sigma \zeta} \, \sigma \wedge \zeta 
+
W^{\ov{\zeta}}_{\sigma \ov{\zeta}} \, \sigma \wedge \ov{\zeta} 
+
W^{\ov{\zeta}}_{\rho \zeta} \, \rho \wedge \zeta \\
+
W^{\ov{\zeta}}_{\rho \overline{\zeta}} \, \rho \wedge \overline{\zeta}
+
W^{\ov{\zeta}}_{\zeta \overline{\zeta}} \, \zeta \wedge \overline{\zeta}
.\end{multline*}

Introducing the modified Maurer-Cartan form $\Lambda$:
\begin{equation*}
\Lambda := \frac{d \A}{\A} + 
\frac{W^{\rho}_{\sigma \rho}}{2} \, \rho - \frac{W^{\sigma}_{\sigma \rho}}{3} \, \sigma - \frac{W^{\sigma}_{\sigma \rho}}{3} \, \zeta 
- \frac{W^{\sigma}_{\sigma \ov{\zeta}}}{3} \, \ov{\zeta}
,\end{equation*}
these equations rewrite in the absorbed form as:
\begin{equation} \label{eq:coframe}
\begin{aligned}
d \sigma & = 3 \, \Lambda \wedge \sigma +
\rho \wedge \zeta
+
\rho \wedge \ov{\zeta}, \\
d \rho &=  2 \, \Lambda \wedge \rho + i \, \zeta \wedge \ov{\zeta}, \\
d \zeta &= \Lambda \wedge \zeta + \frac{I_1}{\A^4} \, \sigma \wedge \rho + \frac{I_2}{\A^3} \, \sigma \wedge \zeta + 
\frac{I_3}{\A^3} \, \sigma 
\wedge \ov{\zeta} + \frac{I_4}{\A^2} \, \rho \wedge \zeta + \frac{I_5}{\A^2} \, \rho \wedge \ov{\zeta}, \\
d \ov{\zeta} &= \Lambda \wedge \ov{\zeta} + \frac{\ov{I_1}}{\A^4} \, \sigma \wedge \rho + 
\frac{\ov{I_3}}{\A^3} \, \sigma \wedge \zeta + 
\frac{\ov{I_2}}{\A^3} \, \sigma 
\wedge \ov{\zeta} + \frac{\ov{I_5}}{\A^2} \, \rho \wedge \zeta + \frac{\ov{I_4}}{\A^2} \, \rho \wedge \ov{\zeta},
\end{aligned}
\end{equation}
where the invariants $I_{i}, \, \, i=2 \dots 5,$ are given by:
\begin{multline*}
\ov{I_2} = \frac{i}{8} \, \frac{Q \LL(B)^2}{B^{\frac{1}{2}}} - \frac{i}{8} \, B^{\frac{1}{2}} \LL(B) Q^2 - \frac{3i}{4} \, \frac{\LL(\LL(B)) \LL(B)}{B^{\frac{1}{2}}} 
+ \frac{i}{4} \, B^{\frac{1}{2}} \LL(B) \LL(Q)\\  - \frac{i}{2} \, B^{\frac{1}{2}} P \LL(B) - \frac{i}{4} \, B^{\frac{1}{2}} Q \LL( \LL(B)) 
- \frac{i}{4} B^{\frac{1}{2}} Q \LL(\LL(B)) \\- \frac{3i}{4} \, B^{\frac{3}{2}} Q \LL(Q) + \frac{i}{2} \, B^{\frac{3}{2}} P Q + \frac{3i}{8} \, \frac{\LL(B)^3}{B^{\frac{3}{2}}}
+ \frac{i}{8} \, B^{\frac{3}{2}} Q^3 \\ + \frac{i}{2} \, B^{\frac{3}{2}} \LL( \LL(Q)) + \frac{i}{2} \, B^{\frac{1}{2}} \, \LL\left( \LL \left( \LL(B) \right) \right)
- i B^{\frac{3}{2}} \LL(P)
\end{multline*}

\begin{multline*}
I_3 = - {\bf D}_0  {\bf C}_0 + \frac{\LL(B)}{B^{\frac{1}{2}}} \,  {\bf D}_0 + B^{\frac{1}{2}} Q  {\bf D}_0  + \frac{A}{B^{\frac{1}{2}}} \,  {\bf D}_0 
- \frac{\Lb\left( {\bf D}_0 \right)}{B^{\frac{1}{2}}}
- i \,  {\bf B}_0  {\bf D}_0  + i \, {\bf B}^2_0 {\bf C}_0 \\  - \frac{A}{B^{\frac{1}{2}}} \,  {\bf B}_0 {\bf C}_0  +  {\bf B}_0  \, \LL(A)
+B P \, {\bf B}_0   + \frac{\Lb \left({\bf B}_0\right)}{B^{\frac{1}{2}}} \,  {\bf C}_0 + \frac{1}{2} \, \frac{\Lb(B)}{B^{\frac{3}{2}}} \,  {\bf B}_0 {\bf C}_0 
\end{multline*}

\begin{multline*}
\ov{I_4} = \frac{3}{4} \, i \, \frac{\LL(B)^2}{B} 
+ \frac{1}{6} \, i \, \LL(B) \, Q + \frac{11}{36} \, i \, B \, Q^2 - i \, \LL(\LL(B))
- \frac{2}{3} \, i \, B \, \LL(Q) + i\, B P,
\end{multline*}
\begin{multline*}
\ov{I_5} = \frac{i}{3} \, \LL(A) + \frac{i}{3} \, \Lb(Q) - i \, \frac{\Lb(\LL(B))}{B}
 + \frac{5}{12} \, i \frac{\LL(B)^2}{B}  - \frac{i}{3} \, B \LL(Q) \\ + \frac{11}{36} \, i \, B\, Q^2 + i \, B \, P 
+ \frac{2}{3} \, i \, \frac{\LL(\Lb(B))}{B} - \frac{i}{3} \, \LL(\LL(B)) \\ + \frac{i}{3} \, \frac{A \, \LL(B)}{B}
 + \frac{7}{18} \, i\, \LL(B)\, Q   - \frac{i}{9} \, \frac{\Lb(B) \, Q}{B} + \frac{i}{9} \, A \, Q,
\end{multline*}
and $I_1$ is given by:
\begin{equation*}
I_1 = \frac{2 i}{3} \,  \left(I_3\right)_{\zeta} -  \frac{2 i}{3} \, \left( I_2 \right)_{\ov{\zeta}}.
\end{equation*}

The exterior derivative of $\Lambda$ can be determined by taking the exterior derivative of the four equations
(\ref{eq:coframe}), which leads to the so-called Bianchi-Cartan's identities.
For example, taking the exterior derivative of the first equation of (\ref{eq:coframe}), one gets:
\begin{equation*}
0 =  \left[ 3 \, d \Lambda +  \left( \frac{ I_2}{\A^3} + \frac{\ov{I_3} }{\A^3} \right)\left.\rho \wedge \zeta \right.
 + \left( \frac{\ov{I_2}}{\A^3} + \frac{I_3}{\A^3} \right) \left. \rho \wedge \zeta \right. \right] \wedge \sigma
,\end{equation*}
while taking the exterior derivative of the second equation gives:
\begin{equation*}
0 =  \left[ 2 \, d \Lambda - i \frac{ I_1}{\A^4} \left.\sigma \wedge \ov{\zeta} \right.
 + i \, \frac{\ov{I_1}}{\A^4}  \left. \sigma \wedge \zeta \right. \right] \wedge \rho
.\end{equation*}
Eventually we get:
\begin{equation} \label{eq:final}
d \Lambda = \frac{i}{2} \frac{ I_1}{\A^4} \left.\sigma \wedge \ov{\zeta} \right.
- \frac{i}{2} \, \frac{\ov{I_1}}{\A^4}  \left. \sigma \wedge \zeta \right.
- \frac{1}{3}  \left( \frac{ I_2}{\A^3} + \frac{\ov{I_3} }{\A^3} \right)\left.\rho \wedge \zeta \right.
- \frac{1}{3}  \left( \frac{\ov{I_2}}{\A^3} + \frac{I_3}{\A^3} \right) \left. \rho \wedge \zeta \right.
+ \frac{I_0}{\A^4} \left. \sigma \wedge \zeta \right., 
\end{equation}
where $I_0$ is given by:
\begin{equation*}
I_0 := -\frac{1}{2 \A^4} \left( I_1\right)_{\zeta} - \frac{1}{2 \A^4} \, \left( \ov{I_1} \right)_{\ov{\zeta}}
.\end{equation*}

\section{Cartan Connection} \label{connection}
We recall that the model for CR-manifolds belonging to general class ${\sf II}$
is Beloshapka's cubic ${\sf B} \subset \C^3$, which is defined by the equations:
\begin{equation*}
{\sf B}: \qquad \qquad
\begin{aligned}
w_1 & = \ov{w_1} + 2 \, i \, z \ov{z}, \\
w_2 & = \ov{w_2} + 2 \, i \, z \ov{z} \left( z + \ov{z} \right).
\end{aligned}
\end{equation*}
Its Lie algebra of infinitesimal CR-automorphisms is given by the following theorem:
\begin{theorem}{\cite{pocchiola2}.}
\label{thm:Bc}
Beloshapka's cubic,
\begin{equation*}
{\sf B}: \qquad \qquad
\begin{aligned}
w_1 & = \ov{w_1} + 2 \, i \, z \ov{z}, \\
w_2 & = \ov{w_2} + 2 \, i \, z \ov{z} \left( z + \ov{z} \right),
\end{aligned}
\end{equation*}
has a ${\bf 5}$-dimensional Lie algebra of CR-automorphisms
${\sf aut_{CR}}({\sf B})  $. 
A basis for the Maurer-Cartan forms of ${\sf aut_{CR}}({\sf B})$ is provided
by the $5$ differential $1$-forms  $\sigma$, $\rho$, $\zeta$,  $\ov{\zeta}$, $\alpha$,
which satisfy the structure equations:
\begin{equation*}
\begin{aligned}
d \sigma &= 3 \left. \alpha \wedge \sigma \right. + \left. \rho \wedge \zeta \right. 
+ \left.  \rho \wedge \ov{\zeta} \right., \\
d \rho &  =  2 \left. \alpha \wedge \rho \right. + i \, \left. \zeta \wedge \ov{\zeta} \right., \\
d \zeta &= \left. \alpha \wedge \zeta \right., \\
d \ov{\zeta} &= \left. \alpha \wedge \ov{\zeta} \right., \\
d \alpha & = 0.
\end{aligned}
\end{equation*}
\end{theorem}
Let us write $\mathfrak{g}$ instead of ${\sf aut_{CR}}({\sf B})$ for the Lie algebra of inifinitesimal automorphisms of Beloshapka's cubic,
and let $\left(e_{\alpha}, e_{\sigma}, e_{\rho}, e_{\zeta}, e_{\ov{\zeta}} \right)$ be the dual basis of the 
basis of Maurer-Cartan 1-forms: $\left(\alpha, \sigma, \rho, \zeta, \ov{\zeta} \right)$ of $\mathfrak{g}$.
From the above structure equations, the Lie brackets structure of $\mathfrak{g}$ is given by:

\begin{alignat*}{3}
\big[ e_{\alpha} , e_{\sigma} \big] & = -3 \, e_{\sigma}, \qquad \qquad 
& \big[ e_{\alpha} , e_{\rho} \big] & = -2 \, e_{\rho}, \qquad \qquad 
& \big[ e_{\alpha} , e_{\zeta} \big] & = - \, e_{\zeta}, \qquad \qquad \\
\big[ e_{\alpha} , e_{\ov{\zeta}} \big] & = - \, e_{\ov{\zeta}}, \qquad \qquad 
& \big[ e_{\rho} , e_{\zeta} \big] & = -e_{\sigma}, \qquad \qquad 
& \big[ e_{\rho} , e_{\ov{\zeta}} \big] & = - e_{\sigma},\qquad \qquad \\
 \big[ e_{\zeta} , e_{\ov{\zeta}} \big] & = - i \, e_{\rho}, 
\end{alignat*}
the remaining brackets being equal to zero.

We refer to \cite{Kobayashi}, p. 127-128, for the definition of a Cartan connection.
Let $\mathfrak{g}_0 \subset \mathfrak{g}$ be the subalgebra spanned by $e_{\alpha}$,
$\mathfrak{G}$ the connected, simply connected Lie group whose Lie algebra is $\mathfrak{g}$ and $\mathfrak{G}_0$ the connected closed
$1$-dimensional subgroup of $\mathfrak{G}$ generated by $\mathfrak{g}_0$. 
We notice that $\mathfrak{G}_0 \cong G_4$, so that $P^4$ is a principal bundle over $M$ with structure group
$\mathfrak{G}_0$, and that $\dim \mathfrak{G} / \mathfrak{G}_0 = \dim M  = 4$.

Let $\left( \Lambda,  \sigma, \rho, \zeta, \ov{\zeta} \right)$ be the coframe of $1$-forms on $P^4$ 
whose structure equation are given by (\ref{eq:coframe}) -- (\ref{eq:final})
 and $\omega$ the $1$-form on $P$ with values in $\mathfrak{g}$
defined by:
\begin{equation*}
\omega(X):= \Lambda (X) \, e_{\alpha} + \sigma(X) \, e_{\sigma} + \rho(X) \, e_{\rho} + \zeta(X) \, e_{\zeta},
+ \ov{\zeta} (X) \, e_{\ov{\zeta}}, 
\end{equation*} 
for $X \in T_pP^4$.
We have:

\begin{theorem}\label{thm:connection}
$\omega$ is a Cartan connection on $P^4$.
\end{theorem} 

\begin{proof}
We shall check that the following three conditions hold:
\begin{enumerate}
\item{\label{it:vertical} $\omega(e_{\alpha}^*) = e_{\alpha}$, where $e_{\alpha}^*$ is the vertical 
vector field on $P^4$ generated by the action of $e_{\alpha}$, }
\item{\label{it:fibre} $R_{a}^* \, \omega = {\sf Ad}(a^{-1})\,  \omega$ for every $a \in \mathfrak{G}_0$,}  
\item{\label{it:isom} for each $p \in P^4$, $\omega_p$ is an isomorphism 
$T_p P^4 \stackrel{\omega_p}{\longrightarrow} \mathfrak{g}$}.
\end{enumerate}

Condition (\ref{it:isom}) is trivially satisfied as $\left( \Lambda, \sigma, \rho, \zeta, \ov{\zeta} \right)$
is a coframe on $P^4$ and thus defines a basis of $T_p^*P^4$ at each point $p$.

Condition (\ref{it:vertical}) follows simply from the fact that $\Lambda$ is a modified-Maurer Cartan form
on $P^4$: 
\begin{equation*}
\Lambda = \frac{d \A}{\A} + 
\frac{W^{\rho}_{\sigma \rho}}{2} \, \rho - \frac{W^{\sigma}_{\sigma \rho}}{3} \, \sigma - \frac{W^{\sigma}_{\sigma \rho}}{3} \, \zeta 
- \frac{W^{\sigma}_{\sigma \ov{\zeta}}}{3} \, \ov{\zeta},
\end{equation*}
so that 
\begin{equation*}
\omega(e_{\alpha}^*) = \Lambda (e_{\alpha^*}) = e_{\alpha},
\end{equation*}
as 
\begin{equation*}
\sigma(e_{\alpha}^*)= \rho(e_{\alpha}^*)= \zeta(e_{\alpha}^*)= \ov{\zeta}(e_{\alpha}^*)= 0,
\qquad \frac{d \A}{\A} (e_{\alpha}^*) = 1,
\end{equation*}
since $e_{\alpha}^*$ is a vertical vector field on $P^4$.

Condition (\ref{it:fibre}) is equivalent to its infinitesimal counterpart:
\begin{equation*}
\LL_{e_{\alpha}^*} \, \omega = - {\sf ad}_{e_{\alpha}} \omega,
\end{equation*}
where 
$\LL_{e_{\alpha}^*} \, \omega$ is the Lie derivative of $\omega$ by the vector field
$e_{\alpha}^*$ and where ${\sf ad}_{e_{\alpha}}$ is the linear map 
$\mathfrak{g} \rightarrow \mathfrak{g}$ defined by:
${\sf ad}_{e_{\alpha}} (X) = \big[ e_{\alpha}, X \big].$
We determine $\LL_{e_{\alpha}^*} \, \omega$ with the help of Cartan's formula:
\begin{equation*}
\LL_{e_{\alpha}^*} \, \omega = e_{\alpha^*}  \, \lrcorner \,  d \omega + d \left( e_{\alpha}^* \, \lrcorner \, \omega \right)
,\end{equation*}
with
\begin{equation*}
d \left( e_{\alpha}^* \, \lrcorner \, \omega \right) = 0
\end{equation*}
from condition (\ref{it:vertical}).
The structure equations (\ref{eq:coframe})--(\ref{eq:final}) give:
\begin{equation*}
e_{\alpha^*}  \, \lrcorner \,  d \omega =  
\begin{pmatrix}
0 \\
3 \, \sigma \\
2 \, \rho \\
\zeta \\
\ov{\zeta}
\end{pmatrix}
,\end{equation*}
which is easily seen being equal to $ - {\sf ad}_{e_{\alpha}} \omega$ from the Lie bracket structure of
$\mathfrak{g}$.
\end{proof}

From theorem \ref{thm:connection}, the structure equations (\ref{eq:coframe}) and (\ref{eq:final}), and the fact that the invariants
$I_0$ and $I_1$ are expressed in terms of $I_2$, $I_3$, $I_4$, $I_5$, we have:

\begin{theorem}
A CR-manifold $M$ belonging to general class ${\sf II}$ 
is locally biholomorphic to Beloshapka's cubic ${\sf B} \subset \C^3$ if and only if the condition
\begin{equation*}
I_2 \equiv I_3 \equiv I_4 \equiv I_5 \equiv 0
\end{equation*}
holds locally on $M$.
\end{theorem}

\end{document}